\documentclass[10pt,psamsfonts,reqno]{amsart}
\usepackage{amsthm,amsmath,amsfonts,amssymb,amsthm,graphicx,epsf,epstopdf,color,pifont,flafter,bm,amscd,tabu}
\usepackage{graphicx,overpic}
\usepackage[margin=3cm]{geometry}
\usepackage{caption}

\newtheorem {theorem} {Theorem}
\newtheorem {proposition} [theorem]{Proposition}

\newtheorem {lemma}  [theorem]{Lemma}

 \usepackage{enumitem}

\begin{document}

\title[Non-isochronicity on piecewise  Hamiltonian systems]{Non-isochronicity on piecewise  Hamiltonian differential systems with homogeneous nonlinearities}

\author[X. Chen and G. Dong]
{Xiaoyi Chen$^1$ and Guangfeng Dong$^2$}

\address{$^1$ Department of Mathematics, Jinan University, Guangzhou, 510632,  China}
\email{cxy6978@stu2022.jnu.edu.cn}

\address{$^2$ Department of Mathematics, Jinan University, Guangzhou, 510632,  China}
\email{donggf@jnu.edu.cn}




\begin{abstract}
In this paper we prove  the non-isochronicity of $\Sigma$-centers for a class of planar piecewise smooth differential systems with a straight switching line,   whose two sub-systems are Hamiltonian differential systems with a non-degenerated center and  only homogeneous nonlinearities.
\end{abstract}

\maketitle

\section{Introduction and  the main results}

Isochronous phenomena occur in many physical problems and have been studied for hundreds of years since the works of C. Huygens in the 17th century. 
A point $p\in\mathbb{R}^2$ is a \emph{center} of a planar differential system  if it is a singular point having a neighborhood $U$ of $p$ such that all the orbits in $\displaystyle U\backslash\{p\}$ are periodic. For each periodic orbit $\gamma \subset \displaystyle U\backslash\{p\}$ we denote its period by $T(\gamma)$. If $T(\gamma)$ is constant for all $\gamma \subset\displaystyle U\backslash\{p\}$, then $p$ is called an \emph{isochronous center}\index{isochronous center}. 

To detect isochronous centers for general differential systems is a very difficult problem, even for polynomial Hamiltonian differential systems,  it is still open now.
Consider a planar polynomial Hamiltonian differential system of the following form:
\begin{equation}\label{HSH}
\begin{array}{ccrr}
  	\dfrac{d x}{d t} & =  & H_y(x,y),& \\
\dfrac{d y}{d t}  &= &  -H_x(x,y),& \quad (x,y)\in\mathbb{R}^2, \quad t\in\mathbb{R},
\end{array}
\end{equation}
where  $H(x,y)=(x^2 +y^2)/2 +\sum_{k=3}^{n+1}H_{k}(x,y)$ is the   {\it Hamiltonian function},  $H_{k}(x,y)$ is a homogeneous  polynomial  of degree $k$, 
and   $$H_y(x,y)=\frac{\partial H(x,y)}{\partial y},\ \ H_x(x,y)=\frac{\partial H(x,y)}{\partial x}.$$

It is clear that the origin $(0,0)$ is a  non-degenerate center.
A generic  level curve $\gamma_h$  defined by the algebraic equation $H(x,y)=h$,
where $h\in \mathbb{R}$ is sufficiently close to $0$, can be parameterized by the Hamiltonian value $h$, so the periods associated to the closed orbits  can be expressed as a  function $T(h)$ of $h$, called the {\it period function}, which  depends on $h$ analytically in the period annulus of the origin. 
The isochronicity of system \eqref{HSH}
  having low degrees
or special forms is studied in
many papers, such as \cite{AGR,ChZh,cima,Gasull,llibre}  and  the references therein.
When the nonlinearity part is a  homogeneous polynomial,
i.e., 
\begin{equation*}\label{m}
H(x,y)=\frac{x^2+y^2}{2}+  H_{n+1}(x,y),
\end{equation*} 
one of the well known results in \cite{Gasull}
 is that the origin is not isochronous for any $n\geq 1$. In the same paper, the authors also give almost all the information of period function for such systems.

In recent decades, more and more works  focus on the isochronicity of 
piecewise smooth differential  systems (e.g., see \cite{Chen12,Gasull03,Li15,Liu-Wang,novaes2022,Tian15} and  the references therein). Such systems can exhibit more complicated dynamics than the smooth systems usually.
Parallel to the center  in  the classical setting, its counterpart
  in the piecewise differential systems is so-called {\it $\Sigma$-center}, which is a topologically center  whose closed orbits surrounding it consist of trajectories of  some component sub-systems in a neighborhood of the center. The {\it period $T(\gamma)$} of a closed orbit $\gamma$ of a $\Sigma$-center is the sum of times associated to the trajectories of the sub-systems  forming $\gamma$. If $T(\gamma)$ is a constant independent of $\gamma$, then the $\Sigma$-center is called {\it isochronous $\Sigma$-center}.

  In \cite{novaes2022}, the authors prove that the tangential $\Sigma$-centers of piecewise differential systems can not be isochronous. The results in \cite{Liu-Wang} show that, if the $\Sigma$-center of system \eqref{HSH} is formed by two  polynomial potential systems,  then it is not isochronous for any $m,n\geq 1$.

In this paper we shall  study the isochronicity for a class of  piecewise Hamiltonian  systems on $\mathbb{R}^2$.
More precisely, we consider
the following real planar piecewise Hamiltonian system $X$ consisting of two sub-systems $X_{\pm}$ with a switching line $y=0$, i.e.,
\begin{equation*}
	X=\begin{cases}
		X_{+}, \ \ y\geq0,\\
		X_{-}, \ \ y<0,
	\end{cases}
\end{equation*}
where  $X_{\pm}$ have the following forms  
\begin{eqnarray*}\label{sys-X}
	X_{+}:\ \left\{\begin{array}{lcr}
		\dfrac{dx}{dt}&=& -\dfrac{\partial H_{+}(x,y)}{\partial y},\\
		\dfrac{dy}{dt}&=& \dfrac{\partial H_{+} (x,y)}{\partial x},
	\end{array}\right.
	\ \ \ \  X_{-}:\ \left\{\begin{array}{lcr}
		\dfrac{dx}{dt}&=& -\dfrac{\partial H_{-}(x,y)}{\partial y},\\
		\dfrac{dy}{dt}&=& \dfrac{\partial H_{-}(x,y)}{\partial x},
	\end{array}\right.
\end{eqnarray*}
with  Hamiltonian functions 
$$H_{+}(x,y)=\frac{x^2+y^2}{2}+H^{+}_{n+1}(x,y),\ H_{-}(x,y)=\frac{x^2+y^2}{2}+H^{-}_{m+1}(x,y),$$ respectively.
Here $H^{+}_{n+1}(x,y)$ and $H^{-}_{m+1}(x,y)$ are  homogeneous polynomials of degree $n+1$ and $m+1$ respectively.
Up to a change of coordinates $(x,y)\mapsto (-x,-y)$, we  can assume $1\leq m\leq n$.
In this paper we shall prove the following main result, which is a counterpart of the non-isochronicity of smooth Hamiltonian systems with only homogeneous nonlinearities.
\begin{theorem}\label{th-m}
For system $X$ the origin is not an isochronous $\Sigma$-center for any $n\geq 1$ and $m\geq 1$. 
\end{theorem}

Theorem \ref{th-m} is also valid for any switching line of the form $ax+by=0$, due to that the forms of systems $X_{\pm}$ are invariant under a rotation of coordinates. 
It is not difficult to see that 
Theorem \ref{th-m} is a direct corollary of the following three propositions, which need different methods to prove.  

\begin{proposition}\label{pro-1}
For system $X$, if $n$ is even, then the origin is not an isochronous $\Sigma$-center. 
\end{proposition}

\begin{proposition}\label{pro-2}
For system $X$,	 if $n $ is odd and $m$ is even,  then the origin is not an isochronous $\Sigma$-center. 
\end{proposition}

\begin{proposition}\label{pro-3}
For system $X$,	 if both $n $ and $m$ are odd,  then the origin is not an isochronous $\Sigma$-center. \end{proposition}

In Section \ref{sec-pre} we shall present some preliminary  results on the period function, and in Section \ref{sec-proof} we give the proof of the main results.

\section{Preliminaries}\label{sec-pre}

Given a system $X$ we can define the {\it correspondence maps} or {\it half Poincar\'e maps}
\begin{equation*}
	\begin{array}{ccccc}
		\varphi_{\pm}: & I_{+}\subseteq \mathbb{R}_{+} & \rightarrow &\mathbb{R}_{-},\\
		& p &\mapsto & \varphi_{\pm}(p),
	\end{array}
\end{equation*}
in a small neighborhood $\Omega_0\subseteq \mathbb{R}^2$ of the origin filled with the closed orbits of $X$,  where $\mathbb{R}_{\pm}$ represent the positive and negative half $x$-axis, $I_{+}=\Omega_0\cap \mathbb{R}_{+}$, and $\varphi_{+}(p)$ (resp. $\varphi_{-}(p)$) is the intersection  between $\mathbb{R}_{-}$ and the connected trajectory of system $X_{+}$ (resp. $X_{-}$) starting (resp. ending) at $p$ in the upper  half plane $\mathbb{R}^2_{+}$ (resp. the lower half plane $\mathbb{R}^2_{-}$).
The origin is a $\Sigma$-center if and only if $\varphi_{+}(p)=\varphi_{-}(p)$  for any $p\in I_{+}$. Furthermore  $\varphi_{\pm} $ can be extended analytically to a segment $I(\supseteq I_{+})$ containing the origin   for polynomial Hamiltonian differential systems $X_{\pm}$.

Note that $p\in \mathbb{R}_{+}\cap I$ can be parameterized by the Hamiltonian value $h$. To avoiding fractional power we use $h^2/2$ in stead of $h$ below. That is, the level curve is reparameterized  by $h$ taking advantage of the equation $H_{\pm}(x,y)=h_{\pm}^2/2$ with $h_{\pm}\geq 0$. 
Then the so-called {\it half period functions} depending on the parameters $h_{\pm}  $, denoted by $T_{\pi \pm}(h_{\pm})$, are defined by 
\begin{equation*}
	\begin{array}{lclcl}
		T_{\pi+}(h_{+})&=&\dfrac{1}{2h_{+}}\dfrac{d}{d h_{+}}\left(\displaystyle\int_{0}^{\pi } {r^2(h_{+},\theta)}d \theta\right)&=& \displaystyle\int_{0}^{\pi } \dfrac{d \theta}{1+(n+1)g_{+}(\theta)r^{n-1}(h_{+},\theta)},\\ \\
		T_{\pi-}(h_{-})&=&\dfrac{1}{2h_{-}}\dfrac{d}{d h_{-}}\left(\displaystyle\int_{\pi }^{2\pi } {r^2(h_{-},\theta)}d \theta \right)&=&\displaystyle\int_{\pi}^{2\pi } \dfrac{d \theta}{1+(m+1)g_{-}(\theta)r^{m-1}(h_{-},\theta)},
	\end{array}
\end{equation*} 
where $(r,\theta) $ is the polar coordinates such that $(x,y)=(r\cos\theta, r\sin\theta)$, 
\begin{equation}\label{eq-rh}
	{h_{+}^2}=r^2 + 2 r^{n+1} H^{+}_{n+1}(\cos\theta, \sin\theta),\ {h_{-}^2}=r^2 + 2 r^{m+1} H^{-}_{m+1}(\cos\theta, \sin\theta),
\end{equation} 
and 
$$g_{+}(\theta):=H^{+}_{n+1}(\cos\theta, \sin\theta),\ \ g_{-}(\theta):=H^{-}_{m+1}(\cos\theta, \sin\theta).$$

 Regarding $r$ as a function of $h_{\pm}$ and $\theta$, we can expand it to be  Taylor series with respect to $h_{\pm}$, i.e., 
\begin{equation}\label{eq-hr0}
	r(h_{\pm},\theta) =h_{\pm}\left(1+\sum_{k=2}^{+\infty}b_{(k-1)_{\pm}}(\theta) h_{\pm}^{k-1}\right),
\end{equation}
 where coefficients $\{b_{(k-1)_{\pm}}(\theta)\}$ are to be determined by the  equalities \eqref{eq-rh}.
 
By comparing the coefficients of two sides of the first one  in equalities \eqref{eq-rh},  we have the following equalities.
\begin{eqnarray*}
	\begin{array}{rclrcl}
		b_{(k-1)_+}&=&0, \ \ \ n-1 \nmid \  k-1,&
		b_{(n-1)_{+}}&=& -g_{+}(\theta),\\
		b_{2(n-1)_{+}}&=&\frac{1}{2}(2n+1)g_{+}^2(\theta),&
		b_{3(n-1)_{+}}&=&-\frac{1}{2}(3n^2+2n )g_{+}^3(\theta),\\
		b_{4(n-1)_{+}}&=&\frac{1}{24}(64n^3+48 n^2 -4n -3)g_{+}^4(\theta),&
		b_{j(n-1)_{+}}&=& \lambda_{j_{+}}(n) g_{+}^j(\theta), \ \forall j\in \mathbb{N},
	\end{array}
\end{eqnarray*}
where $\lambda_{j_{+}}(n) $ is a polynomial of $n$. In fact $\lambda_{j_{+}}(n) $ is a nonzero polynomial of degree $j-1$ from the following assertion: the degree of $\lambda_{j_{+}}(n) $ is $j-1$ and the   coefficient $ \lambda_{j_{+},j-1}$ of the leading term $n^{j-1}$ satisfies $ \lambda_{j_{+},j-1}>0$ for even $j$ and $ \lambda_{j_{+},j-1}<0$ for odd $j$. To prove this assertion, we first observe that
\begin{eqnarray}\label{eq-bjn}
	\begin{array}{rcl}
	b_{j(n-1)}&=&-\frac{1}{2} \sum\limits_{\substack{i_1 +i_2 =j \\ 1\leq i_1,i_2<j}} b_{i_1 (n-1)_{+}}b_{i_2 (n-1)_{+}}\\
	 &&- g_{+}\sum\limits_{s=1}^{j-1}\dfrac{(n+1)!}{s!(n+1-s)!}\left(\sum\limits_{\substack{i_1+\cdots +i_s=j-1\\1\leq i_1,...,i_s<j-1}}b_{i_1 (n-1)_{+} }\cdots b_{i_s (n-1)_{+} }\right)	.
	\end{array}
\end{eqnarray}
Then the assertion can be obtained by induction for $j$. More precisely, suppose that the assertion is true for all integers less than  $j$, then
   $\lambda_{i_{1+} }(n)\cdots\lambda_{i_{s+} }(n)$ has a degree ${i_1 +\cdots +i_s-s}$ 
 when $i_1+\cdots +i_s=j-1$ and $1\leq i_1,...,i_s<j-1$ for $1\leq s <j$.
Note that  $$\frac{(n+1)!}{s!(n+1-s)!}=\frac{n^s}{s!}+l.o.t,$$ where $l.o.t$ represents the terms of lower degree with respect to $n$.
Thus by comparing the coefficients of $n^{j-1}$ of two sides of the equation \eqref{eq-bjn} one can obtain 
\begin{eqnarray}\label{eq-lam}
	\begin{array}{l}
	\lambda_{j_{+},j-1}=-  \sum\limits_{s=1}^{j-1}\left(\dfrac{1}{s!}\left(\sum\limits_{\substack{i_1+\cdots +i_s=j-1\\1\leq i_1,...,i_s<j-1}}\left(\lambda_{i_{1+}, i_{1}-1 }\cdots \lambda_{i_{s+}, i_{s}-1 }\right)\right)	\right).
	\end{array}
\end{eqnarray}
Under the inductive hypothesis the sign of  $\lambda_{i_{1+}, i_{1}-1 }\cdots \lambda_{i_{s+}, i_{s}-1 }$ depends only on the number of odd numbers in the set $\{i_1,...,i_s\}$ such that $i_1+\cdots+i_s=j-1$, which is odd if $j$ is even and is even if  $j$ is odd. Therefore $\lambda_{j_{+},j-1}\neq 0$ has an opposite sign of $\lambda_{(j-1)_{+},j-2}$, since all terms in the right side of equation \eqref{eq-lam} have the same sign and one of them is $-\lambda_{(j-1)_{+},j-2}\neq 0$.

Now we can get the Taylor series of $T_{\pi+}(h_{+})$ with respect to $h_{+}$, i.e.,
\begin{equation}\label{eq-thts}
	\begin{array}{c}
		T_{\pi+}(h_{+})=\dfrac{1}{2h_{+}}\dfrac{d}{d h_{+}}\left(\displaystyle\int_{0}^{\pi } {r^2(h_{+},\theta)}d \theta\right)=\pi +\sum\limits_{j=1}^{\infty}\tilde{\lambda}_{j_{+}}(n)c_{j_{+}}h_{+}^{j(n-1)},
	\end{array}
\end{equation} 
where $ c_{j_{+}}=\int_{0}^{\pi} g_{+}^{j}(\theta) d \theta$ and $\tilde{\lambda}_{j_{+}}(n)$ is a nonzero polynomial of degree $j$ such that the  coefficient of the leading term $n^j$ is equal to $j\lambda_{j_{+},j-1}$ by a direct calculation.

Let  $a_{0_+}=g_{+}(0)$ and $r_0\geq 0$ be a reparameterization of the level curves of $X_{+}$ such that $ h_{+}^2=r_0^2 + 2 a_{0_+}r_0^{n+1} $. Namely, $r_0$ is the value of $r$ at the point $(h_{+},\theta)=(h_{+},0)$. 
By taking
\begin{eqnarray*}
	\begin{array}{lcl}
		h_{+}&=&r_0(1+2 a_{0_+} r_0^{n-1})^{\frac{1}{2}}\\&=&r_0\left(1+a_{0_+} r_0^{n-1}-\frac{1}{2}a^2_{0_+} r_0^{2(n-1)} +\frac{1}{2}a^3_{0_+} r_0^{3(n-1)} -\frac{5}{8}a^4_{0_+} r_0^{4(n-1)} +o\left(r_0^{4(n-1)}\right)\right)
	\end{array}
\end{eqnarray*}
into $T_{\pi +}(h_{+})$, we obtain the expression of the half period function  depending on $r_0$, i.e.,
\begin{equation*}\label{eq-Tr0}
	T_{\pi+,r}(r_0):=T_{\pi +}(h_{+}(r_0))=\pi + \sum_{j=1}^{+\infty}\mu_{j(n-1)_{+}} r_0^{j(n-1)},
\end{equation*}
where the coefficients $\{\mu_{j(n-1)_{+}}\}$ satisfy the following equations by recursion,
\begin{eqnarray*}\label{eq-up}
	\begin{array}{rcl}
		\mu_{(n-1)_{+}}&=& -(n+1)c_{1_{+}}, \\
		\mu_{2(n-1)_{+}}&=&\left(n+1\right)\left(2n c_{2_{+}}-(n-1)a_{0_+}c_{1_{+}} \right),\\
		\mu_{3(n-1)_{+}}&=& -\frac{1}{2}(n+1)\left( (9n^2-1) c_{3_{+}} -8(n^2-n)a_{0_+}c_{2_{+}}+(n^2-4n+3)a^2_{0_+}c_{1_{+}}\right), \\
		\mu_{4(n-1)_{+}}&=& \frac{1}{12}(n+1)\left(32n(4n^2-1)c_{4_{+}} -18(n-1)(9n^2-1)a_{0_+}c_{3_{+}}\right.\\
		&&\left. \ \ \ \ \ \ \ \ \ \ \ \ \ +48 n(n-1)(n-2)a^2_{0_+}c_{2_{+}}- 2(n-1)(n^2-8n+15 )a^3_{0_+}c_{1_{+}}\right),\\
		\mu_{j(n-1)_{+}}&=& \sum_{i=1}^{j}q_{j_{+},i}(n)a^{j-i}_{0_{+}}c_{i_{+}}, \ \ \forall j\in\mathbb{N},
	\end{array}
\end{eqnarray*}
 here $q_{j_{+},j}(n)=\tilde{\lambda}_{j_{+}}(n) $ is a nonzero polynomial of $n$.

For the half period functions $T_{\pi-}(h_{-})$ and $T_{\pi-,r}(r_0)$  associated to system $X_{-}$
we have the similar results to $T_{\pi+}(h_{+})$ and $T_{\pi+,r}(r_0)$  by replacing subscript $+$, number $n$ and $c_{j_{+}}$ with subscript $-$, number $m$ and $ c_{j_{-}}=\int_{\pi}^{2\pi} g_{-}^{j}(\theta) d \theta$ respectively.

Finally we get  the period function $T(r_0)$ of system  $X$ depending on $r_0$, i.e., 
$T(r_0)=T_{\pi +,r}(r_0) +T_{\pi -,r}(r_0).$

When systems $X_{\pm}$ are treated as smooth systems in $\mathbb{R}^2$, the whole period functions $T_{\pm }(h_{\pm })$  depending on  $h_{\pm }$ can be expressed as  
\begin{equation*}
	\begin{array}{lclcl}
		T_{+}(h_{+})&=&\dfrac{1}{2h_{+}}\dfrac{d}{d h_{+}}\left(\displaystyle\int_{0}^{2\pi } {r^2(h_{+},\theta)}d \theta\right)&=&\displaystyle\int_{0}^{2\pi } \dfrac{d \theta}{1+(n+1)g_{+}(\theta)r^{n-1}(h_{+},\theta)},\\ \\
		T_{-}(h_{-})&=&\dfrac{1}{2h_{-}}\dfrac{d}{d h_{-}}\left(\displaystyle\int_{0 }^{2\pi } {r^2(h_{-},\theta)}d \theta \right)&=&\displaystyle\int_{0}^{2\pi } \dfrac{d \theta}{1+(m+1)g_{-}(\theta)r^{m-1}(h_{-},\theta)}.
	\end{array}
\end{equation*} 
Then by the same method to obtain the Taylor series of $T_{\pi\pm}(h_{\pm})$ and $T_{\pi\pm,r }(r_0)$, we can get the Taylor series  of $T_{\pm}(h_{\pm})$ and $T_{\pm,r}(r_0):=T_{\pm}(h_{\pm}(r_0))$   by replacing $c_{j_{\pm}}$ with $\int_{0}^{2\pi}g_{\pm}^{j}(\theta)d\theta$  for all $0\leq j \in\mathbb{Z}$ respectively.

Theorems A and C in \cite{Gasull} give    the following properties of $T_{+}(h_{+})$ (or  $T_{-}(h_{-})$) for system $X_{+}$ (or $X_{-}$) defined in $\mathbb{R}^2$.

\begin{proposition}[Theorems A and C in \cite{Gasull}]\label{pro-smoothpf}
\ 
\begin{enumerate}
  \item[$(i)$] For odd $n$,  $T_{+}(h_{+})$ is monotonic decreasing if and only if  $g_{+}(\theta)\geq 0$ for any $\theta\in [0,2\pi]$. Furthermore, only in this case the origin is a global center of $X_{+}$. In other cases the origin has a bounded period annulus and $T_{+}(h_{+})$  either is monotonic increasing to the infinity, or has exactly one critical value which is minimum.
  \item[$(ii)$] For even $n$,  $T_{+}(h_{+})$ is monotonic increasing to the infinity and the origin has a bounded period annulus.
\end{enumerate}
	
\end{proposition} 

\section{Proof of the main results}\label{sec-proof}

First we need to describe when the origin can be a $\Sigma$-center of system $X$.
\begin{lemma}\label{lem-center}
The origin is a $\Sigma$-center of $X$ if and only if one of the following conditions holds.
\begin{itemize}
  \item[$(\mathrm{I})$.] $n,m$  are odd.
  \item[$(\mathrm{II})$.] $n$  is even, $m$  is odd, and $a_{0_{+}}=0$. 
  \item[$(\mathrm{III})$.] $n$  is odd, $m$  is even, and $a_{0_{-}}=0$. 
   \item[$(\mathrm{IV})$.] $n$ and $m$  are even, $n\neq m$, and $a_{0_{+}}=a_{0_{-}}=0$.
  \item[$(\mathrm{V})$.] $n=m$  are even, and $a_{0_{+}}=a_{0_{-}}$.
\end{itemize}
	
\end{lemma}
\begin{proof}
Denote by $r_{1_{\pm}}$ the value of $r$ associated to $\varphi_{\pm}((r_0,0))$ for a  point $(r_0,0)\in I_{+}$. That is, in the polar  coordinates $(r,\theta)$, we have $\varphi_{+}((r_0,0))=(r_{1_{+}},\pi)$ and $\varphi_{-}((r_0,2\pi))=(r_{1_{-}},\pi)$. By equations \eqref{eq-rh},  the following equalities hold,
 \begin{eqnarray}\label{eq-r0r1-1}
	\begin{array}{lcl}
		r_0^2 + 2 a_{0_{+}} r_0^{n+1} &=& r_{1_{+}}^2 + 2 (-1)^{n+1} a_{0_{+}} r_{1_{+}}^{n+1},\\ r_0^2 + 2 a_{0_{-}} r_0^{m+1} &= &r_{1_{-}}^2 + 2 (-1)^{m+1} a_{0_{-}} r_{1_{-}}^{m+1}.
	\end{array}
\end{eqnarray}

The origin is a  $\Sigma$-center if and only if $r_{1_{+}}= r_{1_{-}}$ for any $(r_0,0)\in I_{+}$,  which is equivalent to the following  equality  
\begin{equation}\label{eq-r0r1-2}
	a_{0_{+}} (r_0^{n+1}- (-1)^{n+1}r_{1_{+}}^{n+1})= a_{0_{-}} (r_0^{m+1}- (-1)^{m+1}r_{1_{-}}^{m+1}), \ \ \forall (r_0,0)\in I_{+}.
\end{equation}

If both $n$ and $m$  are odd integers, then we have $r_{1_{+}}= r_{1_{-}}= r_0$ for any  
$(r_0,0)\in I_{+}$. This is because  $r_{1_{+}}$ (resp. $r_{1_{-}}$) have the form $r_0 +o(r_0)$ and $r_{1_{+}}^{n+1}-r_0^{n+1}=(r_{1_{+}}^{2}-r_0^{2})Q_{+}(r_{1_{+}},r_0)$ (resp. $r_{1_{-}}^{m+1}-r_0^{m+1}=(r_{1_{-}}^{2}-r_0^{2})Q_{-}(r_{1_{-}},r_0)$), where $Q_{+}(r_{1_{+}},r_0)$ (resp. $Q_{-}(r_{1_{-}},r_0)$) is a homogeneous polynomial which can not be divided by $ r_{1_{+}}-r_0$ (resp. $ r_{1_{-}}-r_0$). In this case, the origin is always a $\Sigma$-center.

If $n$ is even and $m$ is odd, then we have $r_{1_{-}}=r_0$ for any $(r_0,0)\in I_{+}$, since $r_{1_{-}}= r_0 +o(r_0)$ and $r_{1_{-}}^{m+1}-r_0^{m+1}=(r_{1_{-}}^{2}-r_0^{2})Q_{-}(r_{1_{-}},r_0)$, where $Q_{-}(r_{1_{-}},r_0)$ is a homogeneous polynomial which can not be divided by $ r_{1_{-}}-r_0$. In this case, the origin is  a $\Sigma$-center if and only if $r_{1_{+}}=r_0$ for any $(r_0,0)\in I_{+}$, if and only if $a_{0_{+}} =0$ by equality \eqref{eq-r0r1-2}. Similarly, if $n$ is odd and $m$ is even, then  the origin is  a $\Sigma$-center if and only if  $a_{0_{-}} =0$.

If both $n$ and $m$ are even, then  the equality  \eqref{eq-r0r1-2} becomes 
$$ a_{0_{+}} (r_0^{n+1}+r_{1_{+}}^{n+1})= a_{0_{-}} (r_0^{m+1}+r_{1_{-}}^{m+1}).$$
Noticing that $r_{1_{\pm}}= r_0 +o(r_0)$, when
$n\neq m$,  the equality  \eqref{eq-r0r1-2} holds for any $(r_0,0)\in I_{+}$ if and only if $a_{0_{+}}=a_{0_{-}} =0$. While for  $n=m$, if $r_{1_{+}}\equiv r_{1_{-}}$, then  $a_{0_{+}}=a_{0_{-}} $ by the equality  \eqref{eq-r0r1-2}.	On the contrary, if $a_{0_{+}}=a_{0_{-}} $, then we have $r_{1_{+}}^2 -r_{1_{-}}^2= a_{0_{+}}(r_{1_{+}}^{n+1} -r_{1_{-}}^{n+1}) $ by the equalities  \eqref{eq-r0r1-1}, which implies that  $r_{1_{+}} = r_{1_{-}}$ for any $(r_0,0)\in I_{+}$. The proof is finished.
\end{proof}

\begin{proof}[Proof of Proposition \ref{pro-1}]
Suppose the origin is an isochronous $\Sigma$-center, then the period function $T(r_0)=T_{\pi +,r}(r_0) +T_{\pi -,r}(r_0)\equiv 2\pi$.	By Lemma \ref{lem-center},   we have the following three possible cases:

Case 1. $m$ is odd and $a_{0_{+}} =0$;

Case 2. $m$ is even and $a_{0_{+}} =a_{0_{-}}=0$;

Case 3. $m=n$ and $a_{0_{+}} =a_{0_{-}}$.

For Case 1, on one hand, since $m<n$ we have the coefficient $\mu_{(m-1)_{-}}=0$  of the term $r_0^{(m-1)}$ in the Taylor series of $T(r_0)$ , i.e. $\int_{\pi}^{2\pi}g_{-}(\theta)d\theta=0$. In this case one can get $$\mu_{2(m-1)_{-}}=2m(m+1)\int_{\pi}^{2\pi}g^2_{-}(\theta)d\theta >0.$$
On the other hand, since $n$ is even, $n-1\neq j(m-1)$ for all $j\in\mathbb{N}$.
 So the coefficient of the term $r_0^{(n-1)}$ is equal to $\mu_{(n-1)_{+}}$ in the Taylor series of $T(r_0)$. Thus  we have $\mu_{(n-1)_{+}}=0$, i.e. $\int_{\pi}^{2\pi}g_{+}(\theta)d\theta=0$.
 Consequently $$T(r_0)=2\pi +\mu_{2(m-1)_{-}} r_0^{2(m-1)}+o\left(r_0^{2(m-1)}\right)\not \equiv 2\pi.$$ This is a contradiction.

For Case 2, since $m\leq n$ and both $n-1$ and $m-1$ are odd, the coefficient of $r_0^{2(m-1)}$ in the Taylor series of $T(r_0)$ is equal to  $\mu_{2(m-1)_{-}}$ as $m<n$ and to $2\mu_{2(m-1)_{-}}$ as $m=n$. However, we have $$\mu_{2(m-1)_{-}}= 2m(m+1) \int_{\pi}^{2\pi}g_{-}^2(\theta)d\theta>0,$$
which contradicts to $T(r_0)\equiv 2\pi$.

For Case 3,  we must have  $\mu_{j(n-1)_{+}}+\mu_{j(n-1)_{-}}=0$ for all $j\geq 1$.   The first equation  
		$\mu_{(n-1)_{+}}+\mu_{(n-1)_{-}}=0$ 	
 implies that 
\begin{eqnarray}\label{eq-lca3}	
\begin{array}{lcl}
		\displaystyle\int_{0}^{\pi}g_{+}(\theta)d\theta+\int_{\pi}^{2\pi}g_{-}(\theta)d\theta=0.
	\end{array}
\end{eqnarray}
However,   the second equation can not hold, since under the equation \eqref{eq-lca3} we have 
$$\mu_{2(n-1)_{+}}+\mu_{2(n-1)_{-}}=2n(n+1)\left(\int_{0}^{\pi}g^2_{+}(\theta)d\theta+\int_{\pi}^{2\pi}g^2_{-}(\theta)d\theta\right)>0.
$$
The proof is finished.
\end{proof}

To prove Propositions \ref{pro-2} and \ref{pro-3}, we first study 
the relations between
the half period functions $T_{\pi\pm}(h_{\pm})$ and the whole period functions $T_{\pm }(h_{\pm})$ of systems $X_{\pm}$ regarded as smooth systems in $\mathbb{R}^2$.

\begin{lemma}\label{lem-pf}
If one of the following conditions 
\begin{itemize}
  \item $n$ is odd $($resp. $m$  is odd$)$,   
  \item  $ \int_{0}^{\pi} g_{+}^{2j-1}(\theta)=0$ $($resp. $ \int_{\pi}^{2\pi} g_{-}^{2j-1}(\theta)=0$$)$ for all $j\in \mathbb{N}$, 
\end{itemize}
holds, then $T_{\pi+}(h_{+})=\frac{1}{2}T_{+}(h_{+})$ $($resp. $T_{\pi-}(h_{-})=\frac{1}{2}T_{-}(h_{-})$ $)$. 
\end{lemma}
\begin{proof}
	If $n$ is odd, then $$\int_{0}^{\pi} g_{+}^{k}(\theta)d \theta=\int_{\pi}^{2\pi} g_{+}^{k}(\theta)d \theta=\frac{1}{2}\int_{0}^{2\pi} g_{+}^{k}(\theta)d \theta$$
	hold for any $k\in \mathbb{N}$, due to that $g_{+}(\theta)$ is a homogeneous polynomial of degree $n+1$ with respect to $\cos\theta$ and $\sin\theta$ and $$\cos(\theta+\pi)=-\cos\theta,\ \ \sin(\theta+\pi)=-\sin\theta.$$  
	The conclusion can be obtained directly from equations \eqref{eq-hr0} and  expressions \eqref{eq-thts} for $T_{\pi+}(h_{+})$ and $ T_{+}(h_{+})$ respectively.
	
If	$ \int_{0}^{\pi} g_{+}^{2j-1}(\theta)=0$ for all $j\in \mathbb{N}$,  then the   conclusion can be gotten from the expressions \eqref{eq-thts} for $T_{\pi+}(h_{+})$ and $ T_{+}(h_{+})$,  and
the following  equalities 
$$ \int_{0}^{\pi} g_{+}^{2k}(\theta)d \theta=\int_{\pi}^{2\pi} g_{+}^{2k}(\theta)d \theta=\frac{1}{2}\int_{0}^{2\pi} g_{+}^{2k}(\theta)d \theta,\ \ \forall k\in\mathbb{N}\cup\{0\}.$$

The same arguments are also valid for 	$T_{\pi-}(h_{-})$ and $T_{-}(h_{-})$.
\end{proof}

The next lemma describes the boundary  $\Gamma$ of the period annulus $\Omega$ of the $\Sigma$-center of system $X$.
 Let $\Gamma_{\pm}$ be  the boundaries  of the period annulus of the origin for systems $X_{\pm }$ defined in the whole plane $\mathbb{R}^2$ respectively. If the origin is an isochronous $\Sigma$-center of system  $X$ and $T_{\pi\pm}(h_{\pm})=T_{\pm}(h_{\pm})/2$, then 
it  can not be global centers simultaneously for systems $X_{\pm}$. Otherwise $T_{\pi\pm}(h_{\pm})$ are monotonic decreasing  simultaneously by Proposition \ref{pro-smoothpf},  thereby $T(r_0)$ is also monotonic decreasing in a neighborhood of the origin, since $h_{\pm}=r_0 +o(r_0)$ are monotonic increasing in a neighborhood of $r_0=0$. Namely, at least one of $\Gamma_{\pm}$ is a non-empty bounded   set, and if $ \Gamma=\Gamma_{+}$ (or $ \Gamma=\Gamma_{-}$) in the upper half plane $\mathbb{R}^2_{+}$ (or the lower half plane $\mathbb{R}^2_{-}$), then $ \Gamma_{+}\neq \emptyset$ (or $ \Gamma_{-}\neq \emptyset$). 

\begin{lemma}\label{lem-boundary}
	If the origin is an isochronous $\Sigma$-center of system $X$ and $T_{\pi\pm}(h_{\pm})=\frac{1}{2}T_{\pm}(h_{\pm})$, then the boundary $\Gamma$ of the period annulus satisfies either $ \Gamma=\Gamma_{+}\neq \emptyset$ in $\mathbb{R}^2_{+}$ or  $ \Gamma=\Gamma_{-}\neq \emptyset$ in $\mathbb{R}^2_{-}$.
\end{lemma}

\begin{proof}
	
	Lemma 2 in \cite{Gasull} states that there is no points $(r,\theta)$ such that $ 1+(n+1)g_{+}(\theta) r^{n-1}=0$ (resp. $ 1+(m+1)g_{-}(\theta) r^{m-1}=0$) in the period annulus associated with the origin for the system $X_{+}$ (resp. $X_{-}$) defined in $\mathbb{R}^2$.
	By this lemma we have $$\frac{\partial H_{+}}{\partial x}=x+(n+1)a_{0_{+}} x^{n}\neq 0,\ \ \frac{\partial H_{-}}{\partial x}=x+(m+1)a_{0_{-}} x^{m} \neq 0,$$ for any point $(x,0)\neq (0,0)$ on $x$-axis in the period annuli.  This means that each closed orbit in the whole period annulus of the origin for system $X_{+}$ (resp. $X_{-}$) is divided  into exact two topological semicircles by $x$-axis. 
	
	By the arguments before Lemma \ref{lem-boundary} and Proposition \ref{pro-smoothpf}, there is at least one of the period annuli for systems $X_{\pm}$ having a  bounded boundary. 
	Therefore the period annulus of the $\Sigma$-center of system $X$ has either the boundary $\Gamma_{+}$ in $\mathbb{R}^2_{+}$ or the boundary $\Gamma_{-}$ in $\mathbb{R}^2_{-}$.
\end{proof}

Below we shall prove Propositions \ref{pro-2} and \ref{pro-3}.

\begin{proof}[Proof of Proposition \ref{pro-2}]
Suppose the origin is an isochronous $\Sigma$-center, then  by Lemma \ref{lem-center} we have $a_{0_{-}}=0 $, hence $r_0=h_{-}$.  
Since $n$ is odd and $m$ is even, the following equalities
	$$T(r_0)=2\pi + \sum_{j=1}^{+\infty}\mu_{j(n-1)_{+}} r_0^{j(n-1)}  + \sum_{j=1}^{+\infty}\mu_{j(m-1)_{-}} r_0^{j(m-1)} \equiv 2\pi$$
	imply that $n-1=2(m-1)$, $$ \mu_{(n-1)_{+}}=-\mu_{2(m-1)_{-}}<0,\ \ \mu_{2(n-1)_{+}}=-\mu_{4(m-1)_{-}}<0, $$ and
	$\mu_{j(m-1)_{-}}=0$ for any odd $j\in\mathbb{N}$, i.e.,	$$ \mu_{j(m-1)_{-}} =q_{j_{-},j}(m) c_{j_{-}}=q_{j_{-},j}(m) \int_{\pi}^{2\pi} g^j_{-}(\theta)d\theta=0.$$
		Note that $q_{j_{-},j}(m)$ is a nonzero polynomial of $m$, so we have $\int_{\pi}^{2\pi} g^j_{-}(\theta)d\theta=0$ for any odd $j\in\mathbb{N}$.
		 From Lemma \ref{lem-pf} we get $$ T_{\pi{\pm}}(h_{\pm })=\frac{T_{\pm}(h_{\pm })}{2}.$$ 
		
By  Proposition \ref{pro-smoothpf}, 
$T_{\pi-,r}(r_0)=T_{\pi-,r}(h_{-})=T_{\pi-}(h_{-})$ is monotonic increasing in the period annulus $\Omega$ of the $\Sigma$-center. So $T_{\pi+,r}(r_0)$ must be monotonic decreasing at least in a neighborhood of the origin.
Consequently for system $X$ the boundary $\Gamma$ of the period annulus  	has only the following 	possible cases by Lemma \ref{lem-boundary}, each of which leads to a contradiction.
\begin{itemize}
  \item $\Gamma\cap{\mathbb{R}^2_{-}}=\Gamma_{-}$. In this case  the period function of system $X$ is unbounded, since $T_{-}(h_{-})$ is monotonic increasing to the infinity as $h_{-}$ tends to the value on the boundary $\Gamma_{-}$ and $T_{+}(h_{+})$ is bounded by Proposition \ref{pro-smoothpf}. This contradicts to $T(r_0)\equiv 2\pi$.

  \item $\Gamma\cap{\mathbb{R}^2_{+}}=\Gamma_{+}$. In this case the origin is not a global center for the system $X_{+}$ defined in $\mathbb{R}^2$, so $T_{\pi+}(h_{+})$ must have a  critical value which is minimum. 
  Due to $$\frac{dh_{+}}{d r_0}=\frac{r_0+(n+1)a_{0_{+}}r_0^{n}}{h_{+}}>0$$
  in the whole period annulus of $X_{+}$ by Lemma 2 in \cite{Gasull}, 
  $ h_{+}(r_0)$ is monotonic increasing on $r_0$ in $\Omega$. Consequently $T_{\pi+,r}(r_0)$  also has a  minimum critical value. However $T_{\pi-,r}(r_0)=T_{\pi-}(h_{-}(r_0))=T_{\pi-}(r_0)$ is monotonic increasing, thus $ T(r_0)$ can not be constant, which is a contradiction. 
\end{itemize}	
The proof is finished.	 
\end{proof}

\begin{proof}[Proof of Proposition \ref{pro-3}]
Suppose the origin is an isochronous $\Sigma$-center, then $T(r_0)\equiv 2\pi$.
Since both $n$ and $m$ are odd, by Lemma \ref{lem-pf}, we have $$ T_{\pi_{\pm}}(h_{\pm })=\frac{T_{\pm}(h_{\pm })}{2}.$$

For the boundary $\Gamma$  there are the following possibilities by Lemma \ref{lem-boundary}, each of which yields  a contradiction.

\begin{itemize}[
  itemindent = 0pt,
  labelindent = \parindent,
  labelwidth = 2em,
 labelsep = 5pt,
  leftmargin = 2.2cm]
 
  \item[Case 1.] $\Gamma\cap{\mathbb{R}^2_{+}}=\Gamma_{+}$ and $T_{+}(h_{+})$ is monotonic decreasing in a neighborhood of the origin. Since the origin is not a global center for the system $X_{+}$, $T_{+}(h_{+})$ has exact one critical value which is minimum. Besides, note that $h_{+}$ depends on $r_0$ monotonic increasingly in the period annulus of system $X_{+}$, so  $T_{+,r}(r_0)$ also has a minimum  critical value. On the other hand, 
  $T_{-}(h_{-})$ is monotonic increasing in a neighborhood of the origin, thereby it is monotonic increasing on $h_{-}$ in the whole   period annulus $\Omega$ of system $X$. Since $$\frac{dh_{-}}{d r_0}=\frac{r_0+(m+1)a_{0_{-}}r_0^{m}}{h_{-}}>0$$
  in the whole period annulus of system $X_{-}$ by Lemma 2 in \cite{Gasull}, then
  $ h_{-}(r_0)$ is monotonic increasing on $r_0$ in $\Omega$. Consequently  $T_{-,r}(r_0)$ is also monotonic increasing in $\Omega$. This means that the derivative of $T(r_0)=\left(T_{+,r}(r_0)+T_{-,r}(r_0)\right)/2$ is not equal to $0$ at the critical point of $T_{+,r}(r_0)$, i.e., $T(r_0)$ can not be a constant, which is a contradiction.
  
  \item[Case 2.] $\Gamma\cap{\mathbb{R}^2_{+}}=\Gamma_{+}$ and $T_{+}(h_{+})$ is monotonic increasing in a neighborhood of the origin. Then $T_{+}(h_{+})$ tends  to the infinity monotonously as $h_{+}$ tends to the value on $\Gamma_{+}$. However, $T_{-}(h_{-})$ is decreasing and bounded, so $T(r_0)$ is unbounded. This is also a contradiction.
  \item[Case 3.] $\Gamma\cap{\mathbb{R}^2_{-}}=\Gamma_{-}$ and $T_{-}(h_{-})$ is monotonic decreasing in a neighborhood of the origin. This case is similar to the Case 1.
  \item[Case 4.] $\Gamma\cap{\mathbb{R}^2_{-}}=\Gamma_{-}$ and $T_{-}(h_{-})$ is monotonic increasing in a neighborhood of the origin. This case is similar to the Case 2.
\end{itemize}
The proof is finished.
\end{proof}

\section*{Acknowledgements}

This work is supported by  the China Scholarship Council (No. 202306780018).

\end{document}